\documentclass[12pt]{article}  \baselineskip
0.22in \textwidth 16.80cm \textheight 22.0cm \topmargin -0.50cm
\usepackage{appendix}
\usepackage{amsthm,bbm,amssymb,amsmath,bm,mathrsfs}
\usepackage{graphicx}
\usepackage{subfigure}
\usepackage{float}
\usepackage{epsfig}
\usepackage[colorlinks=true, citecolor=blue]{hyperref}
\usepackage{abstract}
\usepackage{natbib}
\usepackage{array}

\usepackage{booktabs}

\usepackage[linesnumbered,ruled,vlined]{algorithm2e}
\SetKwRepeat{Loop}{do}{}

\allowdisplaybreaks[4]
\newtheorem{thm}{\textbf{Theorem}}

\newtheorem{definition}{\textbf{Definition}}
\newtheorem{lem}{\textbf{Lemma}}
\newtheorem{proposition}{\textbf{Proposition}}
\newtheorem{assumption}{\textbf{Assumption}}
\newtheorem{remark}{\textbf{Remark}}
\newtheorem{example}{\textbf{Example}}

\DeclareMathOperator*{\argmax}{argmax}

\newcommand{\CVaR}{\mathrm{CVaR}}
\newcommand{\VaR}{\mathrm{VaR}}

\begin{document}

\title{\vspace{-2.5cm} On the Maximization of Long-Run Reward CVaR for Markov Decision Processes}

\author{{Li Xia$^{1,2}$}, \ {Zhihui Yu$^1$}, \ {Peter W. Glynn$^3$} \\
\small \vspace{-0.4cm}
    {$^1$School of Business, Sun Yat-sen University, Guangzhou, China}\\
\small \vspace{-0.4cm}
    {$^2$Guangdong Province Key Laboratory of Computational Science, Guangzhou, China}\\
    \small \vspace{-0.4cm}
    {$^3$Department of Management Science and Engineering, Stanford University, CA 94305, USA}}
\date{}
\maketitle
\begin{abstract}
This paper studies the optimization of Markov decision processes
(MDPs) from a risk-seeking perspective, where the risk is measured
by conditional value-at-risk (CVaR). The objective is to find a
policy that maximizes the long-run CVaR of instantaneous rewards
over an infinite horizon across all history-dependent randomized
policies. By establishing two optimality inequalities of opposing
directions, we prove that the maximum of long-run CVaR of MDPs over
the set of history-dependent randomized policies can be found within
the class of stationary randomized policies. In contrast to
classical MDPs, we find that there may not exist an optimal
stationary deterministic policy for maximizing CVaR. Instead, we
prove the existence of an optimal stationary randomized policy that
requires randomizing over at most two actions. Via a convex
optimization representation of CVaR, we convert the long-run CVaR
maximization MDP into a minimax problem, where we prove the
interchangeability of minimum and maximum and the related existence
of saddle point solutions. Furthermore, we propose an algorithm that
finds the saddle point solution by solving two linear programs.
These results are then extended to objectives that involve
maximizing some combination of mean and CVaR of rewards
simultaneously. Finally, we conduct numerical experiments to
demonstrate the main results.
\end{abstract}
\textbf{Keywords:} Markov decision process; risk-seeking; CVaR
maximization; minimax theorem; saddle point problem

\section{Introduction}
Risk plays a significant role in decision-making of operations and
management. Consider a decision-maker who is managing an enterprise
that is exposed to risk from systemic factors that cannot be
directly controlled through the actions taken by the enterprise
itself. For example, the enterprise may be significantly impacted by
climate change disruptions, supply chain breakdowns due to a
pandemic or tsunami, or even the possibility of war. Given the
uncontrollable nature of these impacts, the decision-maker might
then reasonably choose to maximize the long-run average profit per
unit time, while disregarding small probability of negative outcomes
that are likely induced by uncontrollable factors. This leads, in
the setting of sequential decision-making, to the consideration of a
Markov decision process (MDP) in which the decision-maker seeks to
maximize the long-run average value of the conditional expectation
$\mathbbm E[R_t | R_t \ge q_t]$, where $R_t$ is the reward generated
in period $t$ and $q_t$ is the $\alpha$'th quantile of $R_t$ (with
$\alpha$ being chosen small). The absolute value of this conditional
expectation coincides with the conditional value-at-risk (CVaR) of
the random variable $R_t$. The central contribution of this paper is
the basic theory and algorithms intended to address this problem of
maximizing long-run average CVaR of rewards, from a risk-seeking
viewpoint.

This MDP optimization criterion is a special type of risk-sensitive
measure. In the related literature, \cite{howard1972} were the first
to study risk-sensitive MDPs. In contrast to classical risk-neutral
MDPs that handle expected discounted or average rewards,
risk-sensitive MDPs reflect the attitude of a decision-maker to
risk, either through risk aversion or through risk seeking.
Risk-sensitive MDPs also have close connections with safe or robust
control in engineering, since incorporation of risk into a problem
formulation can enhance the safety or robustness of systems
\citep{chow2015,Garcia95,Lim13,Wachi20}.

Much of the work on risk-sensitive MDPs focuses on either expected
utility criteria or probability criteria. One version of the
expected utility criterion uses an exponential utility function to
measure the risk of accumulated costs; see for instance,
\cite{howard1972} for discrete-time MDPs and \cite{guo2019} for
continuous-time MDPs. The risk probability criterion aims to
minimize the probability of discounted or average cost that does not
exceed a given target value, examples can be referred to
\cite{huo2017,white1993,wu1999}, etc. Risk-sensitivity also arises
in the setting of variance-related MDP optimality criteria, where
the objective can be to minimize the variance subject to mean
optimality or to minimize metrics that combine the mean and
variance. These variance-related MDP problems have been widely
studied in the literature. Efficient computational algorithms
continue to be actively explored in this setting. Audience can refer
to \cite{sobel1982,sobel1994,xia2016,xia2020} and the references
therein.

CVaR is a risk metric that is widely used within the finance and
engineering fields to address downside risk. It is a coherent risk
measure that is computationally tractable in the setting of
single-stage decision problems; see \cite{Rockafellar00}.
Unfortunately, in the sequential decision-making context, the
dynamic programming principle that leads to conventional Bellman
optimality equations fails when either minimizing or maximizing
CVaR. However, the pioneering paper of \cite{bauerle2011}
investigates a discrete-time MDP with a discounted CVaR minimization
criterion. By using state augmentation, they convert the discounted
CVaR MDP into an ordinary MDP with an enlarged state space. The
existence of an optimal Markov deterministic policy for
finite-horizon problem and an optimal stationary deterministic
policy for infinite-horizon problem is proved under
continuity-compactness conditions. By following this state
augmentation idea, many other works, such as
\cite{Haskell15,huang2016,Miller17,Ugurlu17}, have extended the
results of \cite{bauerle2011} to various other MDP settings.

Although the method of state augmentation can transform the
discounted CVaR MDP to a standard MDP, its computation is
intractable caused by continuous state space. Recently, with the
development of reinforcement learning, approximate algorithms are
proposed to optimize the discounted CVaR MDPs by using neural
networks and gradient-based optimization
\citep{chow2014,stanko2019,tamar2015}, which suffer from the
trapping into local optima and slow convergence rates. Another new
way to handle this computing problem is from sensitivity-based
optimization \citep{xia2022}, where a pseudo CVaR definition and
bilevel MDP formulation are proposed and a policy iteration type
algorithm is developed to fast find local optima. Nevertheless,
efficient computation for solving CVaR MDPs is an ongoing and
challenging research area.

In the literature on risk-sensitive MDPs, most works focus on the
perspective of risk averse. However, risk seeking is also an
important feature that a decision-maker may present, such as
casino-goers. On the other hand, according to the \emph{prospect
theory}, a decision-maker usually presents duality of risk
disposition: People are risk averse in the gain frame (positive
prospect), preferring a sure gain to a speculative gamble, but are
risk seeking in the loss frame (negative prospect), tending to
choose a risky gamble rather than a sure loss \citep{Kahneman1979}.
Therefore, it is meaningful to study the optimization problem with a
risk-seeking preference \citep{Armstrong19}. Our motivation also
comes from the management perspective discussed earlier, in which
one is managing an enterprise in the presence of small probability
externalities that may carry large consequences. Such a problem
might arise in managing an endowment for a university. In that
setting, the endowment manager may wish to maximize the expected
return above a nominal payoff level, that is, above the
$\alpha$-quantile of the annual returns distribution. Although
risk-seeking and risk-averse MDPs have commons, such as the same
challenge caused by the failure of dynamic programming principle,
they have essential differences which make the analysis and
optimization approaches different. We aim to study the long-run CVaR
maximization criterion in MDPs from a risk-seeking perspective,
which is not studied in the literature. Our investigation can also
complement the theoretical framework for risk-sensitive MDPs.


In this paper, our objective is to find an optimal policy, among all
history-dependent randomized policies, that maximizes the CVaR of
instantaneous rewards over an infinite horizon. Considering that the
average CVaR of rewards may not be well defined under some
history-dependent randomized policies, we define two related
quantities: limsup CVaR and liminf CVaR, which correspond to the
long-run CVaR in terms of best case and worst case, respectively. We
give numerical examples that show that the limsup and liminf CVaR of
a history-dependent randomized policy may be different. For
maximizing CVaRs, we establish two optimality inequalities to prove
that the limsup optimum is exactly equal to the liminf optimum and
can be attained within the class of stationary randomized policies.

It is well known that stationary deterministic policies preserve
optimality in classical risk-neutral MDPs \citep{Puterman94}. This
optimality also holds for many risk-averse MDPs
\citep{bauerle2011,Haskell15,xia2020,xia2022}. However, in the
risk-seeking setting of this paper, we find that there exist
counterexamples for which no stationary deterministic policy attains
the maximum CVaR. We show in this paper that there always exists an
optimal stationary randomized policy that requires at most one
randomization, i.e., a policy requiring randomization over at most
two actions. By using an alternative convex optimization
representation of CVaR, we prove that the long-run CVaR maximization
MDP within the class of stationary randomized policies can be
transformed into a saddle point problem of bilevel MDPs with a
minimax form. With the use of the \emph{von Neumann minimax
theorem}, we prove the interchangeability of the minimum and maximum
and establish the existence of saddle point solutions. Furthermore,
we devise an approach that transforms the saddle point problem into
two linear programs by using the special structure of the
convex-concave function. The optimality and complexity analysis of
this algorithm is also discussed. We further extend all the results
to a general scenario of mean-CVaR optimization where the long-run
mean and CVaR are maximized simultaneously. Not surprisingly, we
observe that our long-run CVaR MDP degenerates into an ordinary
long-run average MDP when the probability level is set at 0.
Finally, we conduct numerical experiments that illustrate our main
results.

The contributions of this paper are three-fold. First, we study the
long-run CVaR maximization MDP from a risk-seeking perspective. We
propose a very general MDP problem setting over history-dependent
randomized policies and prove the optimality of stationary
randomized policies. This complements a more complete theoretical
framework on risk-sensitive MDPs with CVaR metrics. Second, we
discover that the optimality of stationary deterministic policies
does not hold for maximizing CVaR, which is contrary to the
optimality of deterministic policies widely existing in risk-neutral
and risk-averse MDPs. We further prove the existence of an optimal
stationary randomized policy that requires at most one
randomization. Third, we convert the long-run CVaR maximization MDP
into a minimax problem and propose an algorithm to efficiently
compute the saddle point solution via linear programming.

The rest of the paper is organized as follows. In
Section~\ref{sec:pro}, we rigorously define our long-run CVaR
optimality criterion for MDPs. In Section~\ref{sec:exi}, we
investigate the structural properties of this long-run CVaR
maximization MDP, including the optimality of stationary randomized
policies and the transformation to a minimax saddle point problem.
In Section~\ref{sec:alg}, we propose a linear programming approach
to solve the CVaR maximization MDP and study its algorithmic
properties. In Section~\ref{sec:exten}, we extend our results to the
mean-CVaR maximization of MDPs. Numerical studies to illustrate our
main results are conducted in Section~\ref{sec:exp}. Finally, we
conclude the paper and discuss some future research topics in
Section~\ref{sec:con}.

\section{Preliminaries and Definitions}\label{sec:pro}
In this section, first we introduce the preliminaries about CVaR
metrics and MDPs. Then we give a fairly rigorous definition of
long-run CVaR optimality criterion in discrete-time MDPs.

\subsection{Conditional Value-at-Risk (CVaR)}\label{sec:cvr}
CVaR is a risk measure that corresponds to the conditional
expectation of losses (or gains) exceeding a given value-at-risk
($\VaR$). CVaR was originally proposed in finance and has been
widely used in other fields, such as energy, manufacturing, and
supply chains \citep{Asensio16,Xie18}. Let $\xi$ be a real-valued
and bounded-mean random variable with cumulative distribution
function $F(z) = \mathbbm P(\xi \leq z)$. The $\VaR$ of $\xi$ at
probability level $\alpha \in (0,1)$ is also called the
$\alpha$-quantile, i.e.,
\begin{equation}\label{VaR}
    \nonumber
    \VaR_{\alpha}(\xi) := \inf\left\{z\in \mathbbm{R}:F(z)\ge \alpha\right\} =: F^{-1}_{\xi}(\alpha).
\end{equation}

\begin{figure}[htbp]
\centering
\includegraphics[width=.65\columnwidth]{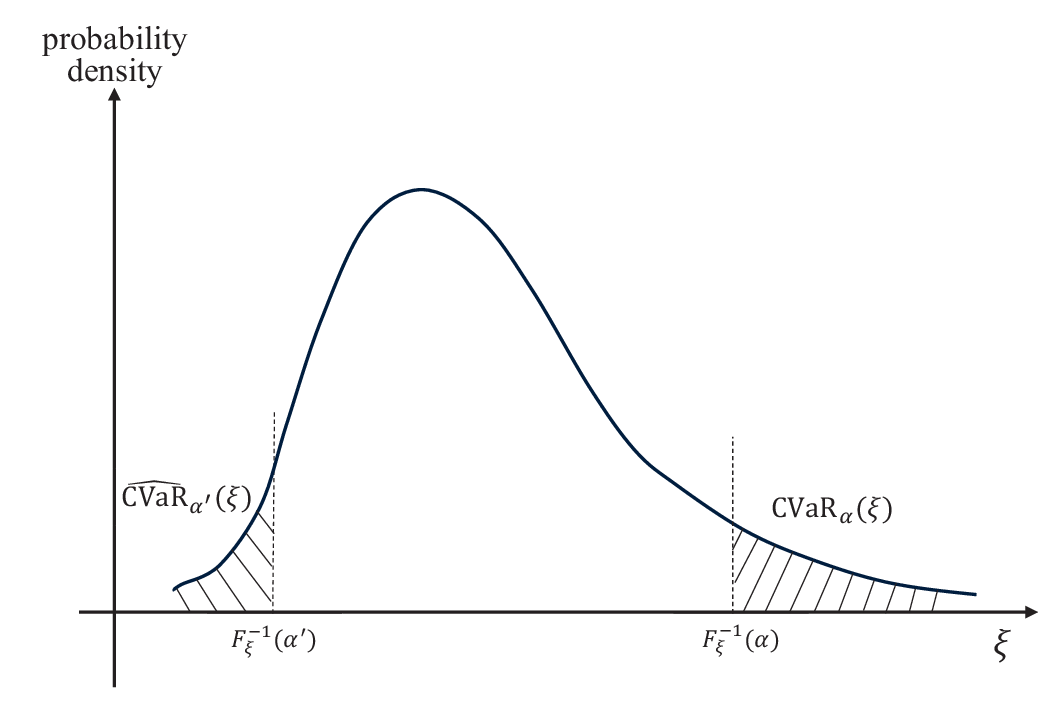}
\caption{Illustration of right-tailed and left-tailed definitions of
CVaR.}\label{fig_CVaR}
\end{figure}

In this paper, the $\CVaR$ of $\xi$ at probability level $\alpha$ is
defined as the expectation of the $(1-\alpha)$-tail distribution of
$\xi$, i.e.,
\begin{equation}\label{eq_CVaR}
    \CVaR_{\alpha}(\xi) := \frac{1}{1-\alpha}\int_{\alpha}^1 \VaR_{q}(\xi)d{q}.
\end{equation}
As illustrated by Fig.~\ref{fig_CVaR}, (\ref{eq_CVaR}) is a
right-tailed definition of CVaR, and there also exists another
definition of CVaR focusing on the left-tailed distribution, i.e.,
\begin{equation}
    \widehat{\CVaR}_{\alpha}(\xi) := \frac{1}{\alpha} \int_0^{\alpha}
    \VaR_{q}(\xi)d{q}. \nonumber
\end{equation}
Obviously, we can derive the following relations
\begin{eqnarray}
\CVaR_{\alpha}(\xi) = -\widehat{\CVaR}_{1-\alpha}(-\xi), \nonumber\\
(1-\alpha)\CVaR_{\alpha}(\xi) + \alpha \widehat{\CVaR}_{\alpha}(\xi)
= \mathbbm E[\xi]. \nonumber
\end{eqnarray}
Therefore, maximizing (or minimizing) the right-tailed CVaR is
equivalent to minimizing (or maximizing) the left-tailed CVaR. In
the rest of this paper, we limit our discussion on the maximization
of the right-tailed CVaR defined in (\ref{eq_CVaR}).

Moreover, \cite{Rockafellar02} discovered that $\CVaR$ is equivalent
to solving a convex optimization problem as follows.
\begin{equation}
    \nonumber
    \CVaR_\alpha(\xi)=\min_{y\in\mathbbm{R}}\left\{y+\frac{1}{1-\alpha}\mathbbm{E}[\xi-y]^+\right\},
\end{equation}
where $[\xi-y]^+=\max\left\{\xi-y,0\right\}$, and
$y^*=\VaR_{\alpha}(\xi)$ exactly attains the above minimum. Note
that, if $\xi$ has lower bound $L$ and upper bound $U$, we can
further specify the domain $y \in \mathbbm{R}$ to a bounded set
$[L,U]$, i.e.,
\begin{equation}\label{equ:CVaR}
    \CVaR_\alpha(\xi)=\min_{y\in[L,U]}\left\{y+\frac{1}{1-\alpha}\mathbbm{E}[\xi-y]^+\right\}.
\end{equation}

As aforementioned, we focus on the maximization of CVaR in this
paper, which reflects the risk-seeking preference of decision
makers. Risk-averse optimization has been widely studied in the
literature. Meanwhile, risk seeking is also an important attitude of
decision makers \citep{Armstrong19}, but much less research
attention has been paid to on this topic. According to the prospect
theory \citep{Kahneman1979}, decision makers usually present
risk-seeking attitude in the loss frame (negative prospect). As
illustrated in Fig.~\ref{fig_riskseeking}, people tend to choose a
risky gamble (Case A) rather than a sure loss (Case B). Obviously,
Case~A has a larger value of CVaR. Such kind of risk-seeking
behaviors can also be found among casino-goers who pursue extreme
rewards although the expectation is negative. Even in reinforcement
learning, an optimization goal including risk-seeking metrics may
enable the algorithm a stronger exploratory capability to find
better solutions \citep{Dilokthanakul18,Mihatsch02}. However, to the
best of our knowledge, there is no literature on the risk-seeking
optimization of CVaR in dynamic scenarios. Thus, it is of
significance to study the maximization of CVaR in MDPs.

\begin{figure}[htbp]
\centering
\includegraphics[width=.95\columnwidth]{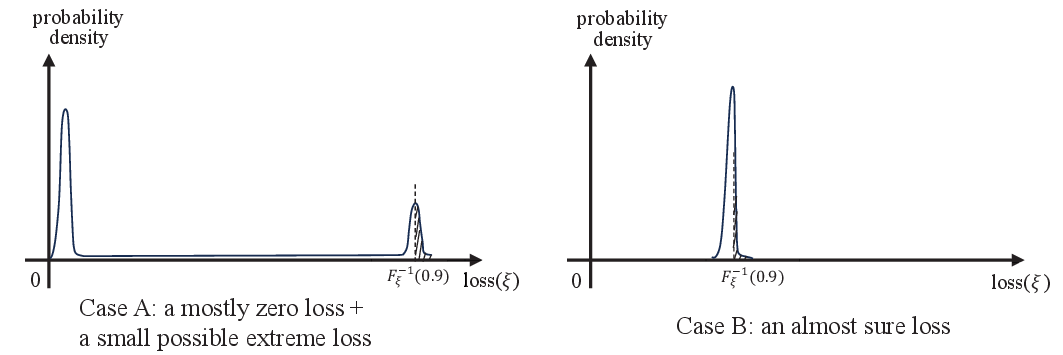}
\caption{Risk-seeking attitude in losses according to the prospect
theory, where Case~A is preferred with a larger CVaR risk
value.}\label{fig_riskseeking}
\end{figure}


\subsection{Markov Decision Process}\label{sec:mdp}
A discrete-time average MDP can be denoted by a tuple $\mathcal M :=
\langle \mathcal S, \mathcal A, (\mathcal A(s), s \in \mathcal{S}),
\bm P, \bm r \rangle$, where $\mathcal S$ and $\mathcal A$ represent
the finite state and action spaces, respectively; $\mathcal A(s)$ is
the set of the admissible actions at state $s$ and we have
$\bigcup\limits_{s \in \mathcal S} \mathcal A(s) = \mathcal A$; $\bm
P$ denotes the Markov kernel, and its element $P(\cdot|s,a)$ is a
probability measure on $\mathcal S$ for each given $(s,a) \in
\mathcal K$ where $\mathcal K := \left\{(s,a) : s \in \mathcal S, a
\in \mathcal A(s)\right\}$; and $\bm r : \mathcal K \rightarrow
\mathbbm{R}$ is the reward function with the minimum $L_r$ and the
maximum $U_r$.

The discrete-time MDP evolves as follows. Suppose the system state
is $s_t \in \mathcal S$ at the current time $t$, and an action $a_t
\in \mathcal A(s_t)$ is adopted based on a policy $u$. The system
will receive an instantaneous reward $R_t := r(s_t,a_t)$, and then
move to a new state $s_{t+1} \in \mathcal S$ at the next time $t+1$
according to the transition probability $P(s_{t+1}|s_t,a_t)$. The
policy $u$ prescribes the action-selection rule at each decision
time epoch based on either history or just the current state, where
the former refers to a \emph{history-dependent randomized policy}
while the latter refers to a \emph{Markov randomized policy}.
Specifically, a history-dependent randomized policy is a sequence of
stochastic kernels, i.e., $u:=(u_{t}, t \geq 0)$, where
$u_t(\cdot|h_t)$ is a probability measure on $\mathcal A$ for given
history $h_t := \left\{s_0,a_0,\ldots,s_{t-1},a_{t-1},s_t\right\}$.
If $u_t(\cdot|h_t)=u_t(\cdot|s_t)$, $\forall h_t$, we call $u$ a
Markov randomized policy which only depends on the current state
$s_t$. Furthermore, if $u_t$ is independent of the decision time
$t$, i.e., there exists a stochastic kernel $d$ on $\mathcal A$
given $\mathcal S$ such that $d(\cdot|s_t)=u_t(\cdot|s_t),\forall t
\ge 0$, we call $d^\infty := (d,d,\cdots)$ or simply $d$ a
\emph{stationary randomized policy}. For notational simplicity, we
denote the sets of all the history-dependent randomized policies,
the Markov randomized policies, and the stationary randomized
policies by $\mathbbm U$, $\mathbbm U^{\rm MR}$, and $\mathbbm D$,
respectively.

For each initial state $s \in \mathcal S$ and policy $u \in \mathbbm
U$, by the Theorem of C. Ionescu-Tulcea
\citep[P.178]{Hernandez-Lerma1996}, there exists a unique
probability measure $\mathbbm{P}_{s}^{u}$ on the space of
trajectories of the states and actions such that
$\mathbbm{P}_{s}^{u}(s_0,a_0,s_1,a_1,\ldots)=\delta_s(s_0)u_0(a_0|s_0)\\P(s_1|s_0,a_0)u_1(a_1|s_0,a_0,s_1)\cdots$,
where $\delta$ denotes the Dirac measure.

\subsection{Definition of Long-Run CVaR Criterion}\label{sec:lr_cvr}
In this paper, we focus on the long-run CVaR criterion in
discrete-time MDPs. We give a rigorous definition as follows, where
the probability level $\alpha \in (0,1)$ is assumed fixed.
\begin{definition}\label{def:lrcvr}
For each initial state $s \in \mathcal S$ and policy $u \in \mathbbm
U$, let $R^{s,u}_t$ be the instantaneous reward at time $t$, the
limsup long-run CVaR is defined as
\begin{equation}\label{equ:lrc+}
    \CVaR^u_{+}(s) := \limsup\limits_{T\rightarrow \infty}{\frac{1}{T}{\sum\limits_{t=0}^{T-1}}} \CVaR_{\alpha} \big(R^{s,u}_t\big),
\end{equation}
and the liminf long-run CVaR is defined as
\begin{equation}\label{equ:lrc-}
    \CVaR^u_{-}(s) := \liminf\limits_{T\rightarrow \infty}{\frac{1}{T}{\sum\limits_{t=0}^{T-1}}} \CVaR_{\alpha} \big(R^{s,u}_t\big).
\end{equation}
If it holds that $\CVaR^u_{+}(s)=\CVaR^u_{-}(s)$ for some $u \in
\mathbbm U$ and each $s \in \mathcal S$, the common function is
called the long-run CVaR and denoted by $\CVaR^u(s)$, i.e.,
\begin{equation}\label{equ:lrc}
    \CVaR^u(s) := \lim\limits_{T\rightarrow \infty}{\frac{1}{T}{\sum\limits_{t=0}^{T-1}}} \CVaR_{\alpha} \big(R^{s,u}_t\big).
\end{equation}
For notational simplicity, we denote by $\hat{\mathbbm U}$ the set
of all policies that make (\ref{equ:lrc}) well defined, i.e.,
$\CVaR^u_{+}(s)=\CVaR^u_{-}(s), \forall u \in \hat{\mathbbm U}$.
\end{definition}

\begin{remark}
\rm{To make the dependence on the initial state $s$ explicit, we
write $R^{s,u}_t$ as the per-step reward, which is a bounded-mean
random variable with support on $\left\{{r}(i,a): (i,a) \in \mathcal
K \right\}$ and the corresponding probability distribution is
$\mathbbm{P}^{u,t}_s(i,a) = \mathbbm{P}^u_s(s_t=i,a_t=a)$.
Considering that the limit of (\ref{equ:lrc}) may not exist (see
Example~\ref{exam1} in Section~\ref{sec:exp} for instance), we
derive two related quantities as (\ref{equ:lrc+}) and
(\ref{equ:lrc-}), which is similar to average MDPs
\citep{Puterman94}. Furthermore, (\ref{equ:lrc}) is well defined for
stationary randomized policies $u \in \mathbbm{D}$, as later
Theorem~\ref{thm:main}~illustrates.}
\end{remark}


With the definitions of \eqref{equ:lrc+}-\eqref{equ:lrc}, we define
the corresponding MDP optimization problems for these CVaR metrics,
respectively.
\begin{definition}\label{def:optimal}
We define the following CVaR maximization problems in MDPs
\begin{eqnarray}\label{equ:optp+}
 \CVaR^{u^*_+}_{+}(s) &=& \sup\limits_{u \in \mathbbm U}\CVaR^u_{+}(s), \qquad \forall s \in \mathcal S, \label{def:supmaxoptim} \\
 \CVaR^{u^*_-}_{-}(s) &=& \sup\limits_{u \in \mathbbm U}\CVaR^u_{-}(s), \qquad \forall s \in \mathcal S, \label{def:infmaxoptim} \\
 \CVaR^{u^*}(s) &=& \sup\limits_{u \in \hat{\mathbbm U}}\CVaR^u(s), \qquad \forall s \in \mathcal S, \label{def:maxoptim}
\end{eqnarray}
where $u^*_+$, $u^*_-$, and $u^*$ are called the limsup-optimal,
liminf-optimal, and optimal policies for the long-run CVaR
maximization problem of MDPs, respectively.
\end{definition}
For notational simplicity, we denote
$\CVaR^{*}_{+}(\cdot):=\CVaR^{u^*_+}_{+}(\cdot)$,
$\CVaR^{*}_{-}(\cdot):=\CVaR^{u^*_-}_{-}(\cdot)$, and
$\CVaR^{*}(\cdot):=\CVaR^{u^*}(\cdot)$. Note that, the policy set
$\hat{\mathbbm U}$ is difficult to determine since it requires
$\CVaR^u_{+}(s)=\CVaR^u_{-}(s)$ for any $u \in \hat{\mathbbm U}$.
Fortunately, we further discover that the optima of these three
problems \eqref{def:supmaxoptim}-\eqref{def:maxoptim} are the same
and can be attained by a common stationary randomized policy $d^*
\in \mathbbm D$.

\section{Existence of Optimal Stationary Randomized Policy}\label{sec:exi}

First, we derive the following lemma to show that the optimality of
Markov randomized policies for these long-run CVaR MDP problems.
\begin{lem}\label{lem:markov}
For any given $s \in \mathcal S$, the optima $\CVaR^{*}_{+}(s)$,
$\CVaR^{*}_{-}(s)$, and $\CVaR^{*}(s)$ in
\eqref{def:supmaxoptim}-\eqref{def:maxoptim} can be attained by
Markov randomized policies in $\mathbbm U^{\rm MR}$, respectively.
\end{lem}
\begin{proof}
We first prove the statement for $\CVaR^{*}(s)$, other statements
for limsup and liminf optimum can be proved with the same argument.

We fix $s \in \mathcal S$. For each given $u \in \mathbbm U$, using
the property (\ref{equ:CVaR}), we have
\begin{align}\label{equ:l1}
        \nonumber
        \CVaR_{\alpha}\big(R^{s,u}_t\big)&=\min_{y\in \mathbbm{Y}}\left\{y+\frac{1}{1-\alpha}\mathbbm{E}[R^{s,u}_t-y]^+\right\}\nonumber\\
        &=\min_{y \in \mathbbm{Y}}\left\{y+\frac{1}{1-\alpha}\sum\limits_{(i,a)\in
        \mathcal K}\mathbbm{P}^{u,t}_s(i,a)[r(i,a)-y]^+\right\},
\end{align}
where $\mathbbm{Y}:=[L_r,U_r]$. With Theorem 5.5.1 of
\cite{Puterman94}, there exists a policy $u' \in \mathbbm U^{\rm
MR}$ such that the associated Markov chains have the same $t$-step
distribution, i.e.,
\begin{equation}\label{equ:l2}
        \mathbbm{P}^{u',t}_s(i,a) = \mathbbm{P}^{u,t}_s(i,a), \quad\forall t \ge 0, ~ (i,a) \in \mathcal K.
\end{equation}
Substituting (\ref{equ:l1}) and (\ref{equ:l2}) into (\ref{equ:lrc}),
we derive
\begin{equation}\label{equ:13}
    \CVaR^{u'}(s)=\CVaR^u(s),
\end{equation}
whenever the RHS (right-hand-side) of (\ref{equ:13}) is well
defined. Thus, Lemma~\ref{lem:markov} holds.
\end{proof}

With Lemma~\ref{lem:markov}, we can restrict the policy searching
space of \eqref{def:supmaxoptim}-\eqref{def:maxoptim} from $\mathbbm
U$ to $\mathbbm U^{\rm MR}$. By imposing the following assumption of
unichain and aperiodicity, we can further restrict our attention to
stationary randomized policy space $\mathbbm D$.

\begin{assumption}\label{ass:ergo}
For each $d \in \mathbbm{D}$, the Markov chain under policy $d$ is
unichain and aperiodic.
\end{assumption}
Assumption~\ref{ass:ergo} indicates that the Markov chain
$\left\{s_t, ~ t \ge 0\right\}$ tends to be steady as $t\rightarrow
\infty$ \citep{Puterman94}. That is, the limiting distribution of
the discrete-time MDP under policy $d\in \mathbbm{D}$ is well
defined
\begin{equation}\label{equ:steady}
    \pi^d(i,a) := \lim\limits_{t \rightarrow \infty} \mathbbm{P}^{d,t}_s(i,a),\quad\forall s \in \mathcal S, ~ (i,a) \in \mathcal K,
\end{equation}
which also satisfies the following stationary distribution equation
\begin{equation}\label{equ:con}
    \begin{cases}
        \sum\limits_{a \in \mathcal A(j)} \pi^d(j,a) = \sum\limits_{(i,a) \in \mathcal K} P(j \mid i, a) \pi^d(i, a), \quad\forall j \in \mathcal S;\\
        \sum\limits_{(i,a) \in \mathcal K}  \pi^d(i, a) = 1; ~~~~~~~
        \pi^d(i,a) \ge 0, \quad \forall (i,a) \in \mathcal K.
    \end{cases}
\end{equation}
For notational simplicity, we denote $\mathbbm{X}$ as the set of
vectors $\bm x \in \mathbbm{R}^{\lvert \mathcal K \rvert}$
satisfying (\ref{equ:con}), where we replace $\pi^d$ with $x$. We
call $\mathbbm X$ the set of feasible solutions of steady-state
distribution of the MDP, and each $\bm x \in \mathbbm X$ indicates
the steady-state distribution of the MDP under a particular
randomized policy $d \in \mathbbm D$.

Since $R^{s,d}_t$ follows the distribution $\mathbbm{P}^{d,t}_s$,
(\ref{equ:steady}) indicates $R^{s,d}_t
\stackrel{L}{\longrightarrow} R^d$ (convergence in distribution),
where $R^d$ is a random variable defined with support on
$\left\{{r}(i,a): (i,a) \in \mathcal K \right\}$ and corresponding
probability distribution $\bm \pi^d$. Next, we derive
Lemma~\ref{lem:equivalent}  to establish the equivalence between
$\CVaR_\alpha (R^{s,d}_t)$ and $\CVaR_\alpha (R^d)$.
\begin{lem}\label{lem:equivalent}
For each initial state $s \in \mathcal S$ and stationary randomized
policy $d\in \mathbbm{D}$, it holds that
\begin{equation}\label{equ:equivalent}
        \lim\limits_{t\rightarrow \infty}\CVaR_\alpha \big(R^{s,d}_t\big)=\CVaR_\alpha \big(R^d\big).
\end{equation}
\begin{proof}

We use the alternative representation (\ref{equ:CVaR}) of $\CVaR$
to estimate the gap between $\CVaR_\alpha (R^{s,d}_t)$ and
$\CVaR_\alpha (R^d)$.
\begin{align*}
            \bigg\lvert \CVaR_\alpha \big(R^{s,d}_t \big)-\CVaR_\alpha \big(R^d \big) \bigg\rvert
            &= \bigg\lvert \min_{y \in \mathbbm{Y}}\left\{y+\frac{1}{1-\alpha}\mathbbm{E}[R^{s,d}_t-y]^+\right\}-\min_{y \in \mathbbm{Y}}\left\{y+\frac{1}{1-\alpha}\mathbbm{E}[R^d-y]^+\right\}\bigg\rvert \\
            & \leq \frac{1}{1-\alpha}\max\limits_{y\in \mathbbm{Y}} \bigg\lvert \mathbbm{E}[R^{s,d}_t-y]^{+}-\mathbbm{E}[R^d-y]^{+}\bigg\rvert \\
            & =\frac{1}{1-\alpha}\max\limits_{y\in \mathbbm{Y}} \bigg\lvert \sum\limits_{(i,a) \in \mathcal K}[\mathbbm{P}^{d,t}_s(i,a)-\pi^d(i,a)][r(i,a)-y]^{+}\bigg\rvert \\
            & \leq\frac{1}{1-\alpha}\max\limits_{y\in \mathbbm{Y}}  \sum\limits_{(i,a)\in \mathcal K}\big\lvert \mathbbm{P}^{d,t}_s(i,a)-\pi^d(i,a)\big\rvert[r(i,a)-y]^{+}\\
            & \leq \frac{U_r-L_r}{1-\alpha} \sum\limits_{(i,a)\in \mathcal K}\big\lvert \mathbbm{P}^{d,t}_s(i,a)-\pi^d(i,a)\big\rvert,
\end{align*}
where the first and second inequalities follow from
$\lvert\min\limits_{y} f_1(y) - \min\limits_{y} f_2(y) \rvert\le
\max\limits_{y} \lvert{f_1(y) - f_2(y)}\rvert$ and the absolute
value inequality, respectively, and the last inequality is ensured
by the fact $[r(i,a)-y]^{+} \leq U_r-L_r$ for any $y \in
\mathbbm{Y},(i,a) \in \mathcal K$. Thus, by using (\ref{equ:steady})
and the finiteness of $\mathcal K$, we see that the gap in the above
inequality goes to 0 as $t \rightarrow \infty$, and
Lemma~\ref{lem:equivalent} holds.
\end{proof}
\end{lem}

\begin{remark}
{\rm Lemma~\ref{lem:equivalent} implies an important property of
CVaR: The order of CVaR measure and the limit of $t$ is
interchangeable, i.e.,
\begin{equation}\label{equ:conver}
  \lim\limits_{t\rightarrow \infty}\CVaR_\alpha \big(\xi_t\big)=\CVaR_\alpha \big(\xi \big),
\end{equation}
if $\left\{\xi_t\right\}$ is a series of bounded discrete-valued
random variables with $\xi_t  \stackrel{L}{\longrightarrow} \xi$.
Furthermore, $(\ref{equ:conver})$ also holds if $\xi_t
\stackrel{P}{\longrightarrow} \xi$ (convergence in probability) or
$\xi_t \stackrel{a.s.}{\longrightarrow} \xi$ (almost sure
convergence) since both of them imply $\xi_t
\stackrel{L}{\longrightarrow} \xi$.}
\end{remark}

Below, we give an important consequence of
Lemma~\ref{lem:equivalent}.
\begin{lem}\label{lemma:equi_sta}
For any initial state $s \in \mathcal S$ and stationary randomized
policy $d \in \mathbbm{D}$, the limsup, liminf, and long-run CVaR
defined in (\ref{equ:lrc+})-(\ref{equ:lrc}) all equal the CVaR of
$R^d$, i.e.,
\begin{equation}\label{equ:thm1}
  \CVaR^d_{+}(s)=\CVaR^d_{-}(s) = \CVaR^d(s) = \CVaR_{\alpha} \big(R^d\big), \qquad\forall s \in \mathcal S, ~ d \in \mathbbm D.
\end{equation}
\end{lem}
\begin{proof}
Since $\lim\limits_{t\rightarrow \infty}\CVaR_\alpha
\big(R^{s,d}_t\big)=\CVaR_\alpha \big(R^d\big)$, the limit of
(\ref{equ:lrc}) exists and further
\begin{equation*}
 \CVaR^d(s)=\lim\limits_{T\rightarrow \infty}{\frac{1}{T}{\sum\limits_{t=0}^{T-1}}} \CVaR_{\alpha} \big(R^{s,d}_t\big)=\CVaR_\alpha \big(R^d \big).
\end{equation*}
Since $\CVaR^d(s)$ in (\ref{equ:lrc}) is well defined as above, it
implies that $\CVaR^d_{+}(s) = \CVaR^d_{-}(s) = \CVaR^d(s)$. Thus,
the lemma is proved.
\end{proof}


Next, we focus on the optimization analysis of the long-run CVaR
maximization in MDPs, as defined in
\eqref{def:supmaxoptim}-\eqref{def:maxoptim}. We find that the
optimality of stationary randomized policies is guaranteed, as
stated by Theorem~\ref{thm:main}.
\begin{thm}\label{thm:main}
The optima $\CVaR^{*}_{+}(s)$, $\CVaR^{*}_{-}(s)$, and
$\CVaR^{*}(s)$ in \eqref{def:supmaxoptim}-\eqref{def:maxoptim} over
the set of history-dependent randomized policies can be attained by
a stationary randomized policy $d^* \in \mathbbm D$, i.e.,
\begin{equation}\label{equ:obj_fun}
\CVaR^{*}_{+}(s) = \CVaR^{*}_{-}(s) = \CVaR^{*}(s) = \CVaR^{d^*}(s),
\quad \forall s \in \mathcal S,
\end{equation}
where $d^*$ is attained by
\begin{equation}\label{equ:obj_pol}
d^* \in \argmax\limits_{d\in \mathbbm{D}}\CVaR_\alpha \big(R^d\big).
\end{equation}
\end{thm}
\begin{proof}
By using Lemma~\ref{lemma:equi_sta}, we just need to prove
\begin{equation}
    \nonumber
    \CVaR^*(s) = \max\limits_{d\in \mathbbm{D}}\CVaR_\alpha \big(R^d\big),\quad\forall s \in \mathcal S,
\end{equation}
which is equivalent to prove
\begin{equation}\label{equ:bine1}
        \CVaR^*_{+}(s) \le \max\limits_{d\in \mathbbm{D}}\CVaR_\alpha \big(R^d\big),\quad\forall s \in \mathcal S,
\end{equation}
and
\begin{equation}\label{equ:bine2}
        \CVaR^*_{-}(s) \ge \max\limits_{d\in \mathbbm{D}}\CVaR_\alpha \big(R^d\big),\quad\forall s \in \mathcal S.
\end{equation}
First, we prove (\ref{equ:bine2}). Since $\mathbbm{D} \subset
\mathbbm U^{\rm MR}$, we have
\begin{equation}\label{equ:leq5}
    \nonumber
    \CVaR^*_{-}(s) = \sup\limits_{u \in \mathbbm U^{\rm MR}}\CVaR^u_{-}(s) \ge \max\limits_{d \in \mathbbm{D}}\CVaR^d_{-}(s) = \max\limits_{d \in \mathbbm{D}}\CVaR_\alpha \big(R^d \big),
\end{equation}
where the first equality follows from Lemma~\ref{lem:markov} and the
last one follows from (\ref{equ:thm1}).

Next, we prove (\ref{equ:bine1}). For any $u \in \mathbbm U^{\rm
MR}$ and $s \in \mathcal S$, using the property (\ref{equ:CVaR}), we
have
\begin{align}\label{equ:leq}
    \CVaR^u_{+}(s)
    &= \limsup\limits_{T\rightarrow \infty}{\frac{1}{T}{\sum\limits_{t=0}^{T-1}}} \CVaR_{\alpha} \big(R^{s,u}_t\big)  \nonumber\\
    &= \limsup\limits_{T\rightarrow \infty}{\frac{1}{T}{\sum\limits_{t=0}^{T-1}}} \min_{y \in \mathbbm{Y}}\left\{y+\frac{1}{1-\alpha}\mathbbm{E}[R^{s,u}_t-y]^+\right\}  \nonumber\\
    &\le \min_{y \in \mathbbm Y}\limsup\limits_{T\rightarrow \infty}{\frac{1}{T}{\sum\limits_{t=0}^{T-1}}} \left\{y+\frac{1}{1-\alpha}\mathbbm{E}[R^{s,u}_t-y]^+\right\}     \nonumber\\
    &= \min_{y \in \mathbbm Y} \left\{ y+\frac{1}{1-\alpha}\limsup\limits_{T\rightarrow \infty}{\frac{1}{T}{\sum\limits_{t=0}^{T-1}}} \mathbbm{E}[R^{s,u}_t-y]^+ \right\} \nonumber,
\end{align}
which yields
\begin{equation}\label{equ:leq1}
\CVaR^*_{+}(s)=\sup\limits_{u \in \mathbbm U^{\rm MR}}
\CVaR^u_{+}(s) \le \min\limits_{y \in
\mathbbm{Y}}\left\{y+\frac{1}{1-\alpha}\sup\limits_{u \in \mathbbm
U^{\rm MR}}\limsup\limits_{T\rightarrow
\infty}{\frac{1}{T}{\sum\limits_{t=0}^{T-1}}}
\mathbbm{E}[R^{s,u}_t-y]^+\right\},
\end{equation}
where the equality follows from Lemma~\ref{lem:markov} and the
inequality follows from $\sup\min\{\cdot\} \leq \min\sup\{\cdot\}$.
 Note that, the term $\limsup\limits_{T\rightarrow
\infty}{\frac{1}{T}{\sum\limits_{t=0}^{T-1}}}
\mathbbm{E}[R^{s,u}_t-y]^+$ in (\ref{equ:leq1}) is a standard MDP
with average criterion for given $y \in \mathbbm{Y}$. By using the
fact that $\mathbbm D$ remains optimal over $\mathbbm U^{\rm M}$ in
average reward MDPs (refer to Theorem~$8.4.5$ of
\citet{Puterman94}), we directly have
\begin{equation*}\label{equ:leq2}
\sup\limits_{u \in \mathbbm U^{\rm MR}}\limsup\limits_{T\rightarrow
\infty}{\frac{1}{T}{\sum\limits_{t=0}^{T-1}}}
\mathbbm{E}[R^{s,u}_t-y]^+
 = \max\limits_{d \in \mathbbm{D}}\lim\limits_{T\rightarrow \infty}{\frac{1}{T}{\sum\limits_{t=0}^{T-1}}} \mathbbm{E}[R^{s,d}_t-y]^+
 = \max\limits_{d \in \mathbbm{D}} \sum\limits_{(i,a) \in \mathcal K} \pi^d(i,a)[r(i,a)-y]^+.
\end{equation*}
Substituting the above equation into \eqref{equ:leq1}, we have
\begin{align}
    \CVaR^*_{+}(s) &\le \min\limits_{y \in \mathbbm{Y}}\left\{y+\frac{1}{1-\alpha}\max\limits_{d \in \mathbbm{D}} \sum\limits_{(i,a) \in \mathcal K} \pi^d(i,a)[r(i,a)-y]^+\right\} \nonumber\\
    &=\min\limits_{y \in \mathbbm{Y}}\max\limits_{d \in \mathbbm{D}}\left\{y+\frac{1}{1-\alpha}\sum_{(i,a)\in \mathcal K}\pi^d(i,a)[r(i,a)-y]^+\right\}.\nonumber
\end{align}
Note that the above $\min\max\{\cdot\}$ is interchangeable by using
Lemma~\ref{lem:exist} shown later, thus
\begin{equation}\label{equ:leq4}
    \nonumber
    \CVaR^*_{+}(s) \le \max\limits_{d \in \mathbbm{D}}\min\limits_{y \in \mathbbm{Y}}\left\{y+ \frac{1}{1-\alpha}\sum\limits_{(i,a) \in \mathcal K} \pi^d(i,a)[r(i,a)-y]^+\right\}=\max\limits_{d \in \mathbbm{D}}\CVaR_\alpha \big(R^d \big).
\end{equation}
The above analysis completes the proof of Theorem~\ref{thm:main}.
\end{proof}

\begin{remark}
{\rm With Lemma~\ref{lem:markov} and Theorem~\ref{thm:main}, the
optimality of Markov randomized policy and stationary randomized
policy is guaranteed, respectively. We can limit our policy
searching space to $\mathbbm U^{\rm MR}$, and eventually to
$\mathbbm D$, which significantly reduces the optimization
complexity. Thus, the original long-run CVaR maximization problems
\eqref{def:supmaxoptim}-\eqref{def:maxoptim} are equivalent to
solving \eqref{equ:obj_pol} over the stationary randomized policy
space $\mathbbm D$.} 
\end{remark}

By using (\ref{equ:CVaR}), we find that solving \eqref{equ:obj_pol}
over $\mathbbm D$ can be transformed to the following mathematical
program
\begin{align}\label{equ:max_CVaR}
    \max\limits_{d\in \mathbbm{D}}\CVaR_\alpha \big(R^d \big)
    &= \max\limits_{d\in \mathbbm{D}}
    \min_{y\in\mathbbm{Y}}\left\{y+\frac{1}{1-\alpha}\mathbbm{E}\big[R^d-y \big]^+\right\}\nonumber\\
    &=\max_{d\in \mathbbm{D}}\min_{y\in\mathbbm{Y}}\sum_{(i,a)\in \mathcal K}\pi^d (i,a)\left\{y+\frac{1}{1-\alpha}[r(i,a)-y]^+\right\}\nonumber\\
    &=\max_{\bm{x}\in \mathbbm{X}}\min_{y\in\mathbbm{Y}}\sum_{(i,a)\in \mathcal K}x(i,a)\left\{y+\frac{1}{1-\alpha}[r(i,a)-y]^+\right\},
\end{align}
where $\mathbbm X$ is determined by (\ref{equ:con}).
For notational simplicity, we define
\begin{equation}\label{equ:spe_form}
    v(\bm{x},y) := \sum_{(i,a)\in \mathcal K}x(i,a)\left\{y+\frac{1}{1-\alpha}[r(i,a)-y]^+\right\}.
\end{equation}
Thus, the CVaR maximization problem $(\ref{equ:max_CVaR})$  can be
simply rewritten as
\begin{equation}\label{equ:saddle}
    \max_{\bm{x} \in \mathbbm{X}}\min_{y\in \mathbbm{Y}} v(\bm{x},y).
\end{equation}
Interestingly, we discover that (\ref{equ:saddle}) is a saddle point
problem by using the \emph{von Neumann minimax theorem}
\citep[Theorem 1.2.3]{Barron2013}, and derive the following lemma

\begin{lem}\label{lem:exist}
There exists at least one saddle point solution of
(\ref{equ:saddle}), and it holds that
    \begin{equation}\label{eq_minmax}
        \max_{\bm{x} \in \mathbbm{X}}\min_{y\in \mathbbm{Y}} v(\bm{x},y)=   \min_{y\in \mathbbm{Y}}\max_{\bm{x} \in \mathbbm{X}} v(\bm{x},y).
    \end{equation}
The common value is denoted as $v^*$. Furthermore, a pair
$(\bm{x}^*,y^*)\in \mathbbm{X}\times \mathbbm{Y}$ is said to be a
saddle point solution if and only if
    \begin{equation}\label{equ:iff}
        \nonumber
        v(\bm{x},y^*)\leq v(\bm{x}^*,y^*)=v^*\leq v(\bm{x}^*,y),\quad \forall (\bm{x},y)\in \mathbbm{X}\times \mathbbm{Y}.
    \end{equation}
\end{lem}
\begin{proof}
Note that both $\mathbbm{X}$ and $\mathbbm{Y}$ are convex, closed
and bounded. In order to apply the von Neumann minimax theorem, we
just need to verify that $v(\bm{x},y)$ is concave in $\bm{x}$ and
convex in $y$. The concavity in $\bm x$ follows directly from the
linearity of $\bm{x}$ in \eqref{equ:spe_form}. The convexity in $y$
is easily verified in \eqref{equ:spe_form} \citep{Rockafellar02}.
Thus, the lemma holds.
\end{proof}

Lemma~\ref{lem:exist} indicates that the long-run CVaR maximization
of MDPs under stationary randomized policies can be reduced to a
solvable saddle point problem \eqref{eq_minmax}, which exactly
motivates the development of optimization algorithms in
Section~\ref{sec:alg}. From the right-hand-side of
\eqref{eq_minmax}, we can see that for any given $y$, the inner
problem $\max\limits_{\bm{x} \in \mathbbm{X}} v(\bm{x},y)$ is
actually a standard MDP. Thus, solving \eqref{eq_minmax} can be
viewed as a \emph{bilevel MDP problem} solving a series of MDPs,
which is computationally intensive. Below, we further discuss the
structural property of optimal policies, which is useful for us to
develop computationally efficient algorithms.

It is well known that the optimality of stationary
\emph{deterministic} policies is guaranteed for classical MDPs with
expected discounted or average criteria \citep{Puterman94}. Such
optimality also holds for some risk-averse criteria including
variance-related criterion \citep{sobel1994,xia2020} and discounted
CVaR criterion \citep{Haskell15,huang2016}. For long-run CVaR
minimization criterion in MDPs, the optimality of deterministic
policies is also proved by \cite{xia2022}. However, this seemingly
universal result cannot be extended to our long-run CVaR
maximization MDP, which is demonstrated by a counterexample, refer
to Example~\ref{exam:not_exist} in Section~\ref{sec:exp}. Thus, we
make the following statement
\begin{proposition}\label{pro:not_exist}
The optimum of the long-run CVaR maximization MDP may not be
attainable by a stationary deterministic policy.
\end{proposition}

In order to further explore the structural property of optimal
policies, we introduce the concept of the ``number of
randomizations" to measure the randomness of policies.
\begin{definition}\label{def:random}
The number of randomizations of state $i \in \mathcal S$ under
policy $d \in \mathbbm{D}$ is $m(i,d)$, if there are exactly
$m(i,d)+1$ actions $a$ in $\mathcal A(i)$ such that $d(a|i)>0$.
Further, the number of randomizations under policy $d$ is defined as
    \begin{equation}
    \nonumber
        n(d):=\sum_{i \in \mathcal S} m(i,d).
    \end{equation}
\end{definition}
Definition~\ref{def:random} was originally proposed by
\cite{Altman1999} to study the structural property of optimal
policies in constrained MDPs. We show the existence of an optimal
stationary randomized policy that requires at most one
randomization, as Theorem~\ref{thm:random} states.
\begin{thm}\label{thm:random}
There exists an optimal stationary randomized policy $d^* \in
\mathbbm{D}$ of the long-run CVaR maximization MDP such that the
number of randomizations under $d^*$ is at most one, i.e., $n(d^*)
\le 1$.
\end{thm}
\begin{proof}
Suppose $(\bm{x}^*,y^*)$ is a saddle point of (\ref{eq_minmax}) and
let $d^* \in \mathbbm{D}$ be the corresponding optimal policy.
Noting that $y=\VaR_\alpha(\xi)$ attains the minimum of
(\ref{equ:CVaR}), we can see that $y^*=\VaR_\alpha(R^{d^*})$. By the
definition of $\VaR$, we have
\begin{equation}\label{equ:VaR_Cd}
 y^*=\inf \left\{z \in \mathbbm{R}: \mathbbm {P}(R^{d^*} \le z) \ge
 \alpha\right\}.
\end{equation}
(\ref{equ:VaR_Cd}) implies two facts:
\begin{equation}\label{equ:cons1}
 \mathbbm {P}(R^{d^*} \le y^*) \ge \alpha,
\end{equation}
\begin{equation}\label{equ:cons2}
 \mathbbm {P}(R^{d^*} \le y^{*-}) < \alpha.
\end{equation}
Obviously, $\VaR_\alpha(R^{d^*})$ must belong to the set
$\left\{r(i,a):(i,a) \in \mathcal K \right\}$. Thus, $y^* \in
\left\{r(i,a):(i,a) \in \mathcal K \right\}$. We sort
$\left\{r(i,a):(i,a) \in \mathcal K \right\}$ in ascending order and
denote the minimum distance between any two adjacent rewards
$r(i,a)$'s as $\delta$. Since $R^{d^*}$ is a discrete random
variable, (\ref{equ:cons2}) is equivalent~to
\begin{equation}\label{equ:cons3}
  \mathbbm {P}(R^{d^*} \le y^*-\delta) < \alpha.
\end{equation}
Now, we represent (\ref{equ:cons1}) and (\ref{equ:cons3}) in
expectation forms by using indicator function as below
\begin{equation}\label{equ:cons11}
 \nonumber
   \sum\limits_{(i,a) \in \mathcal K}x^*(i,a)\mathbbm{1}\left\{r(i,a) \le y^*\right\} \ge \alpha,
\end{equation}
\begin{equation}
 \nonumber
   \sum\limits_{(i,a) \in \mathcal K}x^*(i,a)\mathbbm{1}\left\{r(i,a) \le y^*-\delta\right\} < \alpha.
\end{equation}
Given $y^*$, the optimal policy $d^*$ (corresponding to $\bm x^*$)
can be regarded as an optimal solution of the following linear
program
\begin{align}\label{equ:cop}
&\max_{\bm{x}} \quad \sum_{(i,a)\in \mathcal K}x(i,a)\left\{y^*+\frac{1}{1-\alpha}[r(i,a)-y^*]^+\right\} \nonumber\\
& \begin{array}{r@{\quad}l@{}l@{\quad}l}
s.t.& -\sum\limits_{(i,a) \in \mathcal K}x(i,a)\mathbbm{1}\left\{r(i,a) \le y^*\right\} \le -\alpha\\
&\sum\limits_{(i,a) \in \mathcal K}x(i,a)\mathbbm{1}\left\{r(i,a) \le y^*-\delta\right\}+x_0 = \alpha \\
& \bm{x} \in \mathbbm{X}, ~ x_0 > 0,
\end{array}
\end{align}
where $x_0$ is a slack variable. 
The structure of the optimal solution can be analyzed from the
perspective of linear programming as follows.

First, Theorem~\ref{thm:main} implies that linear program
(\ref{equ:cop}) is feasible. It is worth noting that linear program
(\ref{equ:cop}) contains at most $\lvert \mathcal S \rvert + 2$
independent constraints. This implies that there exists an optimal
solution $\tilde{\bm x}^*:=\left\{\bm{x}^*, x^*_0\right\}$ for this
linear program that has at most $\lvert \mathcal S \rvert + 2$
non-zero elements. Recall that $x^*_0$ must be positive when it
attains optimum, it follows that $\bm {x}^*$ has at most $\lvert
\mathcal S \rvert + 1$ non-zero elements. Based on the assumption of
ergodic property, we also have $\sum\limits_{a \in \mathcal
A(i)}{x}^*(i,a)>0$ for each $i \in \mathcal S$. Thus, there exists
at most one state $i_0 \in \mathcal S$ such that ${x}^*(i_0,a_1)>0$
and ${x}^*(i_0,a_2)>0$ for two different actions $a_1,a_2 \in
\mathcal A(i_0)$. Suppose $d^*$ is the optimal policy corresponding
to this $\bm x^*$. According to Definition~\ref{def:random}, the
number of randomizations under $d^*$ is at most one.
\end{proof}

Note that the necessity of the randomization of optimal policies can
be understood to increase the diversity of system behaviors, which
is better for risk seeking. This randomization will vanish if we
consider CVaR minimization instead of maximization, since
deterministics is better for reducing risk. As a consequence of
Theorem~\ref{thm:random}, the number of randomizations under $d^*$
must be $0$ or $1$. When $n(d^*)=0$, it indicates stationary
deterministic policies can preserve the optimality, which brings
advantages for algorithmic study since the searching space is
finite. When $n(d^*)=1$, we have to consider stationary randomized
policies, which brings challenges to algorithms since the searching
space is infinite. In the next section, we discuss how to find the
optimal policy for this long-run CVaR maximization MDP, based on the
special form of the saddle point problem \eqref{eq_minmax} and the
structural properties of optimal policies stated by
Theorems~\ref{thm:main}~and~\ref{thm:random}.


\section{Algorithm}\label{sec:alg}

With the results in the previous section, we see that the long-run
CVaR maximization MDPs \eqref{def:supmaxoptim}-\eqref{def:maxoptim}
are equivalent to \eqref{equ:obj_pol}, further equivalent to the
saddle point problem \eqref{eq_minmax}. For solving a general saddle
point problem, we may use subgradient method to do iterative
approximation (see, for instance, \cite{nedic2009} and the
references therein), but it suffers from local convergence. In this
section, considering the special form of the convex-concave function
$v(\bm{x},y)$ in \eqref{equ:spe_form}, we develop a global algorithm
by formulating (\ref{eq_minmax}) as two linear programs.

With Lemma~{\ref{lem:exist}}, the saddle point
problem~(\ref{eq_minmax}) is equivalent to a general dual pair of
mathematical programs as below \citep{stoer1963}.
\begin{equation}\label{equ:primal MP}
\hspace{-4cm} \mbox{(mathematical program): \hspace{1cm} }
\begin{array}{ll}
    &\min\limits _{y,z_1}\quad z_1 \\
    &   \mbox{s.t.} \left\{
    \begin{array}{lr}
        v(\bm x,y) \leq z_1,\quad\forall \bm x\in \mathbbm{X},\\
        y\in \mathbbm{Y}.
    \end{array}
    \right.
\end{array}
\end{equation}

\begin{equation}\label{equ:dual MP}
\hspace{-4cm} \mbox{(dual program): \hspace{2.8cm} }
\begin{array}{ll}
    &\max\limits _{\bm x,z_2}\quad z_2 \\
    &   \mbox{s.t.} \left\{
    \begin{array}{lr}
        v(\bm x,y) \ge z_2,\quad \forall y\in \mathbbm{Y},\\
        \bm x\in \mathbbm{X}.
    \end{array}
    \right.
\end{array}
\end{equation}
Note that, the two mathematical programs are unsolvable due to
nonlinearity of $v(\bm x,y)$ and infinite constraints. However, the
mathematical program (\ref{equ:primal MP}) and dual program
(\ref{equ:dual MP}) can be reduced to two linear programs with
finite constraints based on the property of convex-concave function
$v(\bm x,y)$ in \eqref{equ:spe_form}.

\begin{lem}\label{lem:piecewise}
    The function $y\rightarrow v(\bm x,y)$ is piecewise linear with endpoints $\left\{r(i,a):(i,a)\in \mathcal K \right\}$.
\end{lem}

Lemma~\ref{lem:piecewise} indicates that $v(\bm x,y)$ takes minimum
at the endpoint set $\left\{r(i,a) : (i,a) \in \mathcal K \right\}$
for any given $\bm x\in \mathbbm{X}$. Thus, the dual program
(\ref{equ:dual MP}) can be reduced to
\begin{equation}\label{equ:dual LP}
\begin{array}{ll}
    & \max\limits _{\bm x,z_2}\quad z_2 \\
    & \mbox{s.t.} \left\{
    \begin{array}{lr}
        v(\bm x,r(i,a))\ge z_2,\quad\forall (i,a)\in \mathcal K, \\
        x\in \mathbbm{X}.
    \end{array}
    \right.
\end{array}
\end{equation}

On the other hand, it is obvious that the function $\bm x
\rightarrow v(\bm x,y)$ is linear on $\mathbbm{X}$. Recall that
$\mathbbm{X}$ is a bounded convex polyhedron. By the fundamental
theorem of linear programming, each point of $\mathbbm{X}$ can be
represented by the convex-combination of its vertices. Suppose
$\left\{\bm x^{l} : l=1,2,\dots,L \right\}$ is the vertex set, which
also refers to the basic feasible solution (BFS) set of feasible
region $\mathbbm{X}$. Then for each $\bm x\in \mathbbm{X}$, there
exist constants $\left\{\lambda_l : l=1,2,\dots,L \right\}$ with
$0\leq \lambda_l\leq 1$ and $\sum\limits_{l=1}^L \lambda_l=1$ such
that $\bm x=\sum\limits_{l=1}^L \lambda_l \bm x^l$. Thus, the
mathematical program (\ref{equ:primal MP}) can be reduced to
\begin{equation}\label{equ:primal LP}
    \begin{array}{ll}
        &   \min\limits _{y,z_1}\quad z_1 \\
        & \mbox{s.t.} \left\{
        \begin{array}{lr}
            v(\bm x^l,y)\leq z_1,\quad\forall l=1,2,\ldots,L,\\
            y\in \mathbbm{Y}.
        \end{array}
        \right.
    \end{array}
\end{equation}
Taking the specific form of $v(\bm{x},y)$ in (\ref{equ:spe_form}),
(\ref{equ:dual LP}) is obviously a standard linear program with
variables $\bm{x}$ and $z_2$, while (\ref{equ:primal LP}) contains a
nonlinear term $[r(i,a)-y]^+$. It is worth noting that the nonlinear
term can be removed by variable substitution. That is, we introduce
new variables $w(i,a):=[r(i,a)-y]^+$ with added constraints $w(i,a)
\ge 0$, $w(i,a) \ge r(i,a)-y$, $w(i,a) \le
r(i,a)-y+[1-b(i,a)](U_r-L_r)$, $w(i,a) \le b(i,a)(U_r-L_r)$ and
$b(i,a) \in \left\{0,1\right\}$. Through this transformation,
(\ref{equ:primal LP}) is expressed as a mixed-integer linear program
(MILP):
\begin{equation}\label{equ:primal LP11}
\hspace{-1cm} \mbox{(MILP): \hspace{1cm} }
\begin{array}{ll}
    &\min\limits _{y,z_1,\bm{w},\bm{b}}\quad z_1 \\
    &\mbox{s.t.}\quad   \left\{
    \begin{array}{lr}
        \sum\limits_{(i,a)\in \mathcal K} x^l(i,a)\big[y+\frac{1}{1-\alpha}w(i,a)\big]\leq z_1,\quad l=1,2,\ldots,L\\
        w(i,a)\ge r(i,a)-y,\quad\forall (i,a)\in \mathcal K\\
        w(i,a)\ge 0,\quad\forall (i,a)\in \mathcal K\\
        w(i,a) \le r(i,a)-y+[1-b(i,a)](U_r-L_r) ,\quad\forall (i,a)\in \mathcal K\\
        w(i,a) \le b(i,a)(U_r-L_r),\quad\forall (i,a) \in \mathcal K\\
        b(i,a) \in \left\{0,1\right\},\quad\forall (i,a) \in \mathcal K \\
        L_r \leq y \leq U_r
    \end{array}
    \right.
\end{array}
\end{equation}
By noting that the coefficient $\frac{1}{1-\alpha}$ of the nonlinear
term $[r(i,a)-y]^+$ is positive, we can further remove the $0-1$
variables $\left\{b(i,a)\right\}$ as well as the corresponding
constraints in (\ref{equ:primal LP11}), since $w(i,a)$ must take
endpoints $0$ or $r(i,a)-y$ when $z_1$ attains the minimum. Such a
simplification cannot hold when the coefficient is negative. With
this simplification, the original problems \eqref{equ:primal MP} and
\eqref{equ:dual MP} are equivalent to the following two linear
programs:
\begin{equation}\label{equ:primal LP1}
 \hspace{-0.7cm} \mbox{(primal program): \hspace{1cm} }
 \begin{array}{ll}
  &\min\limits _{y,z_1,\bm{w}}\quad z_1 \\
  &\mbox{s.t.}\quad   \left\{
  \begin{array}{lr}
   \sum\limits_{(i,a)\in \mathcal K} x^l(i,a)\big[y+\frac{1}{1-\alpha}w(i,a)\big]\leq z_1,\quad l=1,2,\ldots,L\\
   w(i,a)\ge r(i,a)-y,\quad\forall (i,a)\in \mathcal K\\
   w(i,a)\ge 0,\quad\forall (i,a)\in \mathcal K\\
   L_r \leq y \leq U_r
  \end{array}
  \right.
 \end{array}
\end{equation}

\begin{equation}\label{equ:dual LP1}
\hspace{-0.5cm} \mbox{(dual program): \hspace{-0.1cm} }
\begin{array}{ll}
    &\max\limits _{\bm{x},z_2}\quad z_2 \\
    &\mbox{s.t.}\quad   \left\{
    \begin{array}{lr}
        \sum\limits_{(i,a)\in \mathcal K}x(i,a)\big[r(i_0,a_0)+\frac{1}{1-\alpha}(r(i,a)-r(i_0,a_0))^+\big]\ge z_2, ~~ \forall(i_0,a_0)\in \mathcal K\\
        \sum\limits_{a\in \mathcal A(j)}x(j,a)-\sum\limits_{(i,a)\in \mathcal K}P(j|i,a)x(i,a)=0, \quad \forall j \in \mathcal S \\
        \sum\limits_{(i,a)\in \mathcal K}x(i,a)=1\\
        x(i,a)\ge 0,\quad\forall (i,a)\in \mathcal K\\
    \end{array}
    \right.
\end{array}
\end{equation}

The above analysis shows that the saddle point of (\ref{eq_minmax})
can be computed by solving the linear programs (\ref{equ:primal
LP1}) and (\ref{equ:dual LP1}). Thus, we directly derive the
following Theorem~\ref{thm:com}.
\begin{thm}\label{thm:com}
Suppose the optimal solution of the linear program (\ref{equ:primal
LP1}) is $y^*,\bm z^*_1,\bm{w}^*$ and the optimal solution of the
dual program (\ref{equ:dual LP1}) is $\bm x^*,z_2^*$, then $(\bm
x^*,y^*)$ constructs a saddle point of (\ref{eq_minmax}) with
corresponding value $v^*=z_1^*=z_2^*$. Furthermore, the optimal
stationary randomized policy $d^*$ of the long-run CVaR maximization
MDP is determined by
\begin{equation}\label{equ:optimal}
    \nonumber
 \begin{array}{ll}
d^*(a|i)=     \left\{
    \begin{array}{lr}
    \frac{x^*(i,a)}{\sum\limits_{a' \in \mathcal A(i)}x^*(i,a')}, \quad \text{if} \sum\limits_{a' \in \mathcal A(i)}x^*(i,a')>0 ~\text{for}~ i \in \mathcal{S},\\
        \text{arbitrary probability}, \qquad \text{otherwise}.
    \end{array}
    \right.
\end{array}
\end{equation}
\end{thm}


With Theorem~\ref{thm:com}, we obtain an algorithm to solve the
long-run CVaR maximization MDP by these equivalent linear programs.
In order to further study the computation complexity of solving
linear programs (\ref{equ:primal LP1}) and (\ref{equ:dual LP1}), we
need to analyze the number of constraints. It is obviously that the
number of constraints of (\ref{equ:dual LP1}) is $\lvert \mathcal S
\rvert + 2\sum\limits_{i \in \mathcal S}\lvert \mathcal A(i)
\rvert+1$, which is linear with $|\mathcal S|$ or $|\mathcal A|$. As
a comparison, the number of constraints of (\ref{equ:primal LP1}) is
directly determined by the number of BFSs of (\ref{equ:con}). It can
be seen from \cite{denardo1970} that each BFS corresponds to a
stationary deterministic policy. Hence, the number of BFSs equals
the number of stationary deterministic policies, i.e.,
$\prod\limits_{i \in \mathcal S }\lvert \mathcal A(i) \rvert$. Thus,
the number of constraints of (\ref{equ:primal LP1}) is
$\prod\limits_{i \in \mathcal S }\lvert \mathcal A(i) \rvert + 2
\sum\limits_{i \in \mathcal S }\lvert \mathcal A(i) \rvert+2$, which
grows exponentially with $|\mathcal S|$ or $|\mathcal A|$.

\begin{remark}
The number of constraints of linear program (\ref{equ:primal LP1})
is exponential while the number of constraints of dual program
(\ref{equ:dual LP1}) is polynomial. It seems time-consuming to solve
both the two linear programs (\ref{equ:primal LP1}) and
(\ref{equ:dual LP1}) when the state and action spaces are large.
However, for the purpose to obtain the optimal policy $\bm x^*$ and
$v^*$, solely solving (\ref{equ:dual LP1}) is enough, which is more
computationally efficient.
\end{remark}

\section{Extensions}\label{sec:exten}
In this section, we first aim to extend our results to the mean-CVaR
maximization in MDPs, where the long-run CVaR and the long-run
average reward are maximized simultaneously. We define a combined
metric as follows.
\begin{align}\label{equ:mean_cvar}
    J^u_\beta(s) := \CVaR^u(s)+\beta \eta^u(s)=
    \lim\limits_{T\rightarrow \infty}{\frac{1}{T}{\sum\limits_{t=0}^{T-1}}} \big[\CVaR_{\alpha} \big(R^{s,u}_t\big) +\beta  \mathbbm{E}\big(R^{s,u}_t\big)\big],
\end{align}
whenever the limit exists. We use a coefficient $\beta \ge 0$ to
balance the weights between the CVaR and the long-run average reward
$\eta^u(s)$. Note that, if $\beta=0$, the combined metric
degenerates into the long-run CVaR. Thus, the mean-CVaR MDP is an
extension of the long-run CVaR MDP. The objective is to find a
policy to maximize $J^u_\beta(s)$, i.e.,
\begin{align}
\nonumber
    J^*_\beta(s)= \sup\limits_{u \in \tilde{\mathbbm
    U}}J^u_\beta(s), \qquad \forall s \in \mathcal S,
\end{align}
where $\tilde{\mathbbm U}$ is the set of policies that make
(\ref{equ:mean_cvar}) well defined.

Using the same argument in Section~\ref{sec:exi}, the mean-CVaR
maximization problem has an optimal stationary randomized policy and
can be reduced to a solvable saddle point problem, i.e.,
\begin{equation}\label{equ:mean_saddle}
J^*_\beta(s) = \max\limits_{d \in
\mathbbm{D}}J^d_\beta(s)=\max_{\bm{x}\in
\mathbbm{X}}\min_{y\in\mathbbm{Y}}\sum_{(i,a)\in \mathcal
K}x(i,a)\left\{y+\frac{1}{1-\alpha}[r(i,a)-y]^+ +\beta
r(i,a)\right\}, ~ \forall s \in \mathcal S.
\end{equation}
In addition, Proposition~\ref{pro:not_exist} and
Theorem~\ref{thm:random} still hold.

Similar to Section~\ref{sec:alg}, we can solve the saddle problem
(\ref{equ:mean_saddle}) by linear programming, just with
$\sum\limits_{(i,a)\in \mathcal K}
x^l(i,a)\big[y+\frac{1}{1-\alpha}w(i,a)+\beta r(i,a)\big]$ in lieu
of $\sum\limits_{(i,a)\in \mathcal K}
x^l(i,a)\big[y+\frac{1}{1-\alpha}w(i,a)\big]$ in (\ref{equ:primal
LP1}) and $\sum\limits_{(i,a)\in \mathcal
K}x(i,a)\big[r(i_0,a_0)+\frac{1}{1-\alpha}(r(i,a)-r(i_0,a_0))^+
+\beta r(i,a)\big]$ in lieu of $\sum\limits_{(i,a)\in \mathcal
K}x(i,a)\big[r(i_0,a_0)+\frac{1}{1-\alpha}(r(i,a)-r(i_0,a_0))^+\big]$
in (\ref{equ:dual LP1}). Thus, all the results in
Sections~\ref{sec:pro}-\ref{sec:alg} can be extended to mean-CVaR
maximization MDPs without additional technical difficulties.

Moreover, we discuss another extension that may unify the long-run
CVaR and long-run average optimization of MDPs. From the definition
of CVaR in \eqref{eq_CVaR} and Fig.~\ref{fig_CVaR}, we derive that
the CVaR of a random variable at probability level $\alpha = 0$
equals the expectation of the random variable, i.e., $\CVaR_0(\xi) =
\mathbbm E[\xi]$, where we define $\VaR_0(\xi):=\inf\{\xi\}$.
Therefore, when $\alpha = 0$, the long-run CVaR of an MDP under a
policy $u$ defined in \eqref{equ:lrc} can be rewritten as
\begin{equation}
\CVaR^u(s) = \lim\limits_{T\rightarrow
\infty}{\frac{1}{T}{\sum\limits_{t=0}^{T-1}}} \CVaR_{\alpha}
\big(R^{s,u}_t\big) = \lim\limits_{T\rightarrow
\infty}{\frac{1}{T}{\sum\limits_{t=0}^{T-1}}} \mathbbm E[R^{s,u}_t]
=: \eta^u(s), \quad s \in \mathcal S, \nonumber
\end{equation}
whenever the limits exist, and $\eta^u(s)$ is the long-run average
of the MDP under policy $u$ with initial state $s$. We directly
derive the following remark.

\begin{remark}
When $\alpha = 0$, the long-run CVaR optimization of MDPs is
equivalent to the long-run average optimization of MDPs, where the
later one is well studied in classical MDP theory. Our main results
in Section~\ref{sec:exi} are consistent with those for long-run
average MDPs. Our LP algorithm in Section~\ref{sec:alg} for solving
long-run CVaR MDPs, such as LP in (\ref{equ:dual LP1}), degenerates
into the LP formulation for average MDPs (refer to Chapter~8.8 of
\cite{Puterman94}). Therefore, these observations provide a
validation of our approaches, and also bring us a unified viewpoint
of long-run CVaR MDPs and long-run average MDPs.
\end{remark}

\section{Numerical Experiments}\label{sec:exp}
In this section, we conduct numerical examples to illustrate our
main results. First, we construct an example similar to Example
8.1.1 of \cite{Puterman94} to illustrate that the limit in
(\ref{equ:lrc}) may not exist under some history-dependent
randomized policies.
\begin{example}\label{exam1}
Consider an MDP with two states and two actions for each state. Let
$\mathcal S = \left\{s_1,s_2\right\}, \mathcal
A(s_1)=\left\{a_{11},a_{12}\right\}, \mathcal
A(s_2)=\left\{a_{21},a_{22}\right\}$, the reward function
$r(s_1,a_{11})=r(s_1,a_{12})=2, r(s_2,a_{21})=r(s_2,a_{22})=-2$ and
the transition probability
$P(s_1|s_1,a_{11})=P(s_2|s_1,a_{12})=P(s_1|s_2,a_{21})=P(s_2|s_2,a_{22})=1$.
Consider a Markov policy $u$ which, on starting in $s_1$, remains in
$s_1$ for one period, proceeds to $s_2$ and remains there for three
periods, returns to $s_1$ and remains there for $3^2$ periods,
proceeds to $s_2$ and remains there for $3^3$ periods, etc.
\end{example}
In this example, the reward $R^{s_1,u}_t$ is a constant for given
time $t \ge 0$ (which also can be regarded as a particular random
variable with Dirac distribution). Thus, the CVaR of $R^{s_1,u}_t$
under probability level $\alpha=0.5$ is
\begin{align}
    \nonumber
    \CVaR_{0.5}(R^{s_1,u}_t)=\min\limits_{y \in \mathbbm{R}}\left\{y+2\mathbbm{E}(R^{s_1,u}_t-y)^+\right\} &=
    \begin{cases}
        2,& \frac{3^{2n}-1}{2} \le t \le \frac{3^{2n+1}-1}{2}-1,\\
        -2,& \frac{3^{2n+1}-1}{2} \le t \le \frac{3^{2n+2}-1}{2}-1.
    \end{cases}
\end{align}
Then direct computation shows that
\begin{equation}
    \nonumber
    \CVaR^{u}_{+}(s_1)=2,\quad \CVaR^{u}_{-}(s_1)=-2.
\end{equation}
Thus, the limit in (\ref{equ:lrc}) does not exist, and the CVaR may
not be well defined for history-dependent randomized policies. It is
necessary to define liminf long-run CVaR and limsup long-run CVaR.

Next, we give Example~\ref{exam:not_exist} to demonstrate that an
optimal stationary deterministic policy may not exist in our
long-run CVaR maximization MDP.

\begin{example}\label{exam:not_exist}
Consider an MDP with state space $\mathcal S=\left\{1,2,3\right\}$
and action space $\mathcal A=\left\{1,2,3\right\}$ with $\mathcal
A(1)= \mathcal A(2)=\mathcal A(3)= \mathcal A$. Suppose the reward
function and the transition function take values as in
Table~\ref{tab:not_exist}. The probability level is set as
$\alpha=0.7$.
        \renewcommand{\arraystretch}{1.1}
    \begin{table}[!htbp]
        \newcommand{\tabincell}[2]{\begin{tabular}{@{}#1@{}}#2\end{tabular}}
        \centering
        \fontsize{4}{7}\selectfont
        \caption{The reward function and the transition function.}\label{tab:not_exist}
        \small
        \begin{tabular}{|p{18mm}<{\centering}|p{12mm}<{\centering}|p{12mm}<{\centering}|p{12mm}<{\centering}|p{12mm}<{\centering}|p{12mm}<{\centering}|p{12mm}<{\centering}|p{12mm}<{\centering}|p{12mm}<{\centering}|p{12mm}<{\centering}|}
            \hline
            state & \multicolumn{3}{c|}{1}&\multicolumn{3}{c|}{2}&\multicolumn{3}{c|}{3}
        \\  \hline
             action & \tabincell{c}{$1$}& \tabincell{c}{$2$} & \tabincell{c}{$3$}
             & \tabincell{c}{$1$}& \tabincell{c}{$2$} & \tabincell{c}{$3$}
             & \tabincell{c}{$1$}& \tabincell{c}{$2$} & \tabincell{c}{$3$} \\
            \hline
            $P(1|\cdot,\cdot)$&0.4688    &0.3564    &0.3991   &0.1083    &0.7012    &0.4370 &0.5457  &0.4102  &0.1460     \\ \hline
            $P(2|\cdot,\cdot)$&0.0741    &0.0857    &0.1457   &0.1839    &0.1863    &0.4373 &0.1834  &0.4357  &0.3986 \\ \hline
            $P(3|\cdot,\cdot)$&0.4571    &0.5579    &0.4552   &0.7078    &0.1124    &0.1257 &0.2709  &0.1541 &0.4554  \\ \hline
            $r(\cdot,\cdot)$&5    &69    &13   &94    &4    &71 &77  &70 &39  \\ \hline
        \end{tabular}
    \end{table}
\end{example}
We use linear program (\ref{equ:dual LP1}) to compute a stationary
randomized policy to maximize the long-run CVaR. The optimal policy
$d^*$ is mixed with $d^*(3|1)=d^*(1|2)=1, d^*(1|3)=0.0255,
d^*(3|3)=0.9745$, and the corresponding $\CVaR^*=93.24$. By
evaluating all stationary deterministic policies, we find the
maximum corresponding $\CVaR=92.6675$, which does not attain the
optimum. Thus, there does not exist an optimal stationary
deterministic policy in this example. In addition, the number of
randomizations under this optimal policy $d^*$ is one, which is
consistent with Theorem~\ref{thm:random}.

Finally, we conduct another example about university endowment funds
to demonstrate the numerical computation of mean-CVaR maximization
MDPs to compute an optimal stationary randomized policy.
\begin{example}\label{exam:endow}
Endowments have been becoming critical support to many
not-for-profit institutions. College and university endowments are
collections of funds that support students, staff, and the
institution's mission. As an example of such a long-run CVaR
maximization MDP, we consider a long-term investment project --- a
university endowment fund. Following \cite{merton1993}, the focus of
university endowment fund management is on determining the optimal
portfolio allocation among traded assets of the university's total
wealth. The endowment needs to be divided between several different
assets, such as stocks and bonds. In order to make proper
allocation, the decision-maker needs to model future returns on
these assets. To simplify the model, we set the economic environment
as $0$ or $1$, which represents a bear market and bull market,
respectively. We assume that the economic environment is dynamic and
obeys the Markov property. The endowment manager needs to choose the
endowment proportion allocated to a riskless asset (bond) and a
risky assets (stock) at each investment time from
$\left\{0.2,0.5,0.8\right\}$ based on the current economic
environment and the current holding proportion. The endowment
manager wishes to maximize the average return and the expected
return above a nominal payoff level simultaneously.
\end{example}

We establish a discrete-time MDP model to study this simplified
endowment management problem. The state space is $\mathcal S =
\left\{(x,\omega)\right\}$, where $x \in \left\{0, 1\right\}$
represents the current economic state and $\omega \in \left\{0.2,
0.5, 0.8\right\}$ denotes the holding proportion of stock. The
remaining wealth will be allocated to the riskless asset, and thus,
$\omega_{0} = 1 - \omega$. The action space is $\mathcal A =
\left\{0.2, 0.5, 0.8\right\}$, which  determines the wealth
proportion allocated to stock. For each $t \ge 1$, the holding
proportion of stock at time $t$ equals the wealth proportion
allocated to stock at time $t-1$, i.e., $\omega_{t} = a_{t-1}$. The
Markov kernel is uniquely determined by the transition probability
of economic environment, which is set as $\mathbbm P(0|0)=0.8,
\mathbbm P(1|0)=0.2, \mathbbm P(0|1)=0.3, \mathbbm P(1|1)=0.7$. The
reward function can be computed by
\begin{equation}
    \tilde r(s, a, s') :=1000[(1 - a)r_0 + a r_1(x')  - b \lvert a - \omega \rvert], \nonumber
\end{equation}
where $1000$ (million) represents the total initial endowment funds,
$r_0 = 0.02$ denotes the fixed rate of return for riskless asset,
$r_1$ denotes the rate of return for stock which depends on the next
economic state with $r_1(0)=-0.05, r_1(1)=0.1$, and $b = 0.005$
denotes the rate of transaction cost per unit wealth.

Note that, the reward function depends on the next state, it is
necessary to do some equivalent conversion. Consider an MDP with
future-state-dependent reward $\tilde{r}: \mathcal S \times \mathcal
A \times \mathcal S \rightarrow \mathbbm{R}$, the long-run CVaR is
defined similarly to Definition~\ref{def:lrcvr} by replacing
$R^{s,u}_t$ with $\tilde{R}^{s,u}_t$, where $\tilde{R}^{s,u}_t$ is a
bounded-mean random variable with support on
$\left\{\tilde{r}(i,a,j): (i,a,j) \in \mathcal K \times \mathcal S
\right\}$ and corresponding probability distribution
$\tilde{\mathbbm{P}}^{u,t}_s(i,a,j) = \mathbbm{P}^u_s(s_t=i,a_t=a,
s_{t+1}=j)$. Using the same argument in Section~\ref{sec:exten},
there exists an optimal stationary randomized policy that maximizes
the combined metric of long-run CVaR and long-run average return,
which can be also transformed to a saddle point problem, i.e.,
\begin{align}\label{equ:max_CVaR1}
J^*_\beta &= \max\limits_{d \in \mathbbm{D}} \big[\CVaR^d + \beta
\eta^d \big] = \max\limits_{d \in \mathbbm{D}} \lim\limits_{T\rightarrow \infty}{\frac{1}{T}{\sum\limits_{t=0}^{T-1}}} \big[\CVaR_{\alpha} \big(\tilde R^{s,d}_t\big) +\beta  \mathbbm{E}\big(\tilde R^{s,d}_t\big)\big] \nonumber\\
    &=
\max_{\bm{x}\in \mathbbm{X}}\min_{y\in\mathbbm{Y}}\sum_{(i,a,j)\in
\mathcal K \times \mathcal
S}x(i,a)P(j|i,a)\left\{y+\frac{1}{1-\alpha}[\tilde{r}(i,a,j)-y]^+ +
\beta \tilde{r}(i,a,j)\right\}.
\end{align}
With (\ref{equ:primal LP1}) and (\ref{equ:dual LP1}), the saddle
point of (\ref{equ:max_CVaR1}) can be solved by the following two
linear programs.

The linear program is to compute the optimal VaR:
\begin{align}\label{equ:primal LP2}
    &\min\limits _{y,z_1,\bm w}\quad z_1 \nonumber\\
    &\mbox{s.t.}\quad   \left\{
    \begin{array}{lr}
        \sum\limits_{(i,a)\in \mathcal K} x^l(i,a)\big[y+\frac{1}{1-\alpha}\sum\limits_{j\in \mathcal S}P(j|i,a)w(i,a,j)+ \beta\sum\limits_{j\in \mathcal S}P(j|i,a)\tilde{r}(i,a,j)\big] \leq z_1,\quad l=1,2,\ldots,L\\
        w(i,a,j)\ge \tilde{r}(i,a,j)-y,\quad\forall (i,a,j)\in \mathcal K \times \mathcal S\\
        w(i,a,j)\ge 0,\quad\forall (i,a,j)\in \mathcal K \times \mathcal S\\
        L_r \leq y \leq U_r.
    \end{array}
    \right.
\end{align}
The dual program is to compute an optimal stationary randomized
policy:
\begin{align}\label{equ:dual LP2}
    &\max\limits _{\bm{x},z_2}\quad z_2 \nonumber\\
    &\mbox{s.t.}\quad   \left\{
    \begin{array}{lr}
        \sum\limits_{(i,a)\in \mathcal K}x(i,a)\big[\tilde{r}(i_0,a_0,j_0)+\sum\limits_{j\in \mathcal S}P(j|i,a)[\frac{1}{1-\alpha}(\tilde{r}(i,a,j)-\tilde{r}(i_0,a_0,j_0))^+ + \beta \tilde{r}(i,a,j)]\big] \ge z_2,\\
~~~~~~~~~~~~~~~~~~~~~~~~~~~~~~~~~~~~~~~~~~~~~~~~~~~~~~~~~~~~~~~~~~~~~~~~~~~~~~~~~~~~\quad \forall(i_0,a_0,j_0)\in \mathcal K \times \mathcal S\\
        \sum\limits_{a\in \mathcal A(j)}x(j,a)-\sum\limits_{(i,a)\in \mathcal K}P(j|i,a)x(i,a)=0, \quad\forall j\in \mathcal S \\
        \sum\limits_{(i,a)\in \mathcal K}x(i,a)=1\\
        x(i,a)\ge 0,\quad\forall (i,a)\in \mathcal  K.\\
    \end{array}
    \right.
\end{align}

We set $\alpha=0.9,\beta=0.5$ and use Matlab to solve the linear
programs (\ref{equ:primal LP2}) and (\ref{equ:dual LP2}) for this
example. The optimal VaR is $y^*=84$ with the corresponding
$\CVaR^*=96.84$, and the optimal policy $d^*$ is given in
Table~\ref{tab:opti_expm2}.

\renewcommand{\arraystretch}{1.1}
\begin{table}[htbp]
    \newcommand{\tabincell}[2]{\begin{tabular}{@{}#1@{}}#2\end{tabular}}
    \centering
    \fontsize{4}{7}\selectfont
    \caption{The optimal policy of mean-CVaR maximization.}\label{tab:opti_expm2}
    \label{tab:performance_comparison}
    \small
    \begin{tabular}{|p{18mm}<{\centering}|p{12mm}<{\centering}|p{12mm}<{\centering}|p{12mm}<{\centering}|p{12mm}<{\centering}|p{12mm}<{\centering}|p{12mm}<{\centering}|}
        \hline
        state & \tabincell{c}{$(0,0.2)$}& \tabincell{c}{$(0,0.5)$} & \tabincell{c}{$(0,0.8)$}
        & \tabincell{c}{$(1,0.2)$}& \tabincell{c}{$(1,0.5)$} & \tabincell{c}{$(1,0.8)$} \\
        \hline
        $d^*(0.2|\cdot)$&1    &0    &1   &0    &0    &0    \\ \hline
        $d^*(0.5|\cdot)$&0    &1    &0    &0   &1    &0   \\ \hline
        $d^*(0.8|\cdot)$&0    &0    &0    &1   &0    &1  \\ \hline
    \end{tabular}
\end{table}

From Table~\ref{tab:opti_expm2}, it is observed that the optimal
policy is deterministic in this example. Combined with
Example~\ref{exam:not_exist}, the optimal policy is either
deterministic or mixed with one randomization, which further
verifies Theorem~\ref{thm:random}.

\section{Conclusion}\label{sec:con}
In this paper, we study the analysis and optimization algorithms for
MDPs with a risk-seeking perspective. The objective is to find an
optimal policy among history-dependent randomized policies to
maximize the long-run CVaR of instantaneous rewards over an infinite
horizon. By establishing two optimality inequalities, we prove the
optimality of stationary randomized policies over the set of
history-dependent randomized policies. We also find that there may
not exist an optimal stationary deterministic policy and further
prove the existence of an optimal stationary randomized policy that
requires at most one randomization. Via an alternative
representation of CVaR with a form of convex optimization, we
convert the long-run CVaR maximization MDP into a minimax
formulation for solving saddle points, which initiates an algorithm
by solving linear programs. An extension to mean-CVaR maximization
MDPs is also discussed. Finally, numerical experiments are conducted
to demonstrate our main results.

One of the future research topics is to deal with the discounted
CVaR MDP, where the objective is to maximize the CVaR of total
discounted rewards over an infinite horizon. It is also desirable to
develop an effective algorithm to solve the discounted CVaR MDP,
which is not reported in the literature yet. Another future topic is
to study risk measures in stochastic games, from one decision-maker
in MDPs to multiple in games. One possible scheme is to study the
long-run CVaR optimality criterion in the framework of two-person
zero-sum Markov games. Moreover, the combination of our results with
techniques of reinforcement learning is also a promising research
direction, which can contribute to develop a framework of
data-driven risk-seeking decision making.

{}

\end{document}